\documentclass[reqno, 12pt]{article}

\pdfoutput=1

\usepackage{enumerate}
\usepackage{latexsym}
\usepackage[centertags]{amsmath}
\usepackage{amsfonts}
\usepackage{amssymb}
\usepackage{amsthm}
\usepackage{newlfont}
\usepackage{graphics}
\usepackage{color}
\usepackage{float}
\usepackage{diagbox}
\usepackage{hyperref}
\usepackage{longtable}
\usepackage{rotating}
\usepackage{multirow}
\usepackage{extarrows}
\usepackage[sort,compress,numbers]{natbib}
\usepackage[utf8]{inputenc}

\newtheorem{thm}{Theorem}[section]
\newtheorem{cor}[thm]{Corollary}

\newtheorem{prop}[thm]{Proposition}

\theoremstyle{mydefinition}
\newtheorem{dfn}[thm]{Definition}
\theoremstyle{myremark}
\newtheorem{rem}[thm]{Remark}
\newtheorem{exa}[thm]{Example}

\allowdisplaybreaks[4]

\title{MacMahon's $\Omega_\geq$ operator: A computational framework}

\author{Feihu Liu$^{\color{blue} \dag}$ and Guoce Xin$^{\color{blue} \S}$
\\[2mm]
{\small $^{\color{blue} \dag, \S}$ School of Mathematical Sciences,}\\[-0.8ex]
{\small Capital Normal University, Beijing, 100048, P.R.~China}\\
{\small $^{\color{blue} \dag}$ Email address: feihu.liu@cnu.edu.cn}\\
{\small $^{\color{blue} \S}$ Email address: guoce\_xin@163.com}
}

\date{\today}

\begin{document}

\maketitle

\begin{abstract}
MacMahon introduced partition analysis in his book ``Combinatory Analysis'' as a computational technique for solving problems related to systems of linear Diophantine equations and inequalities. This paper aims to develop a fundamental computational method for MacMahon's partition analysis. As applications, we present simplified computations for ``Han's formula'', the ``$k$-gon partitions problem'', and the ``two-dimensional problem''. Moreover, we apply our method to solve a challenging problem.
\end{abstract}

\noindent
\begin{small}
 \emph{Mathematics subject classification}: Primary 05A15; Secondary 05A17, 11P81.
\end{small}

\noindent
\begin{small}
\emph{Keywords}: MacMahon's partition analysis; Omega operator; Partial fraction decomposition; Elliott rational function; $k$-gon partition.
\end{small}

\section{Introduction}

In enumerative combinatorics, \emph{partition analysis} was first introduced by MacMahon in his famous book, ``Combinatorial Analysis'' \cite{MacMahon19}, published in 1915. In this work, MacMahon primarily investigated plane partitions by introducing the omega operator $\mathrm{\Omega}_\geq$.

We briefly outline MacMahon's partition analysis. Let $\mathbb{F}$ be a field. A formal Laurent series in $\lambda_1,\dots,\lambda_k$ is defined as a series that can be expressed in the form
\[
\sum_{i_1=-\infty}^{\infty}\cdots \sum_{i_k=-\infty}^{\infty} a_{i_1\ldots i_k}\lambda_1^{i_1} \cdots \lambda_k^{i_k},
\]
where each $a_{i_1\ldots i_k}$ is an element of $\mathbb{F}$. For such formal Laurent series, the definition of the omega operator $\mathrm{\Omega}_\geq$ as given in \cite{MacMahon19} is clear:
\begin{dfn}\label{dfn-natural}
The operator $\mathrm{\Omega}_\geq$ acts on a formal Laurent series in $\lambda_1,\dots, \lambda_k$ with coefficients $a_{i_1,\dots ,i_k}$ in $\mathbb{F}$ by
\[
\underset{\geq}{\mathrm{\Omega}}\sum_{i_1=-\infty}^{\infty}\cdots \sum_{i_k=-\infty}^{\infty} a_{i_1\ldots i_k}\lambda_1^{i_1} \cdots \lambda_k^{i_k} = \sum_{i_1=0}^{\infty}\cdots \sum_{i_k=0}^{\infty} a_{i_1\ldots i_k}.
\]
\end{dfn}
In particular, if $F(\lambda)$ is a rational function in $\lambda$ whose series expansion contains no negative power in $\lambda$, then $\mathrm{\Omega}_\geq F(\lambda) = F(1)$.

We illustrate the operator with a simple example. Let $\mathbb{N}$ denote the set of all non-negative integers. Suppose $a_1, a_2 \in \mathbb{N}$ with $a_1 \geq a_2$. Consider the generating function
\[
F(x_1,x_2) = \sum_{a_1 \geq a_2 \geq 0} x_1^{a_1} x_2^{a_2}.
\]
Introducing a variable $\lambda$, we obtain
\begin{equation}\label{omegact1}
F(x_1,x_2) = \underset{\geq}{\mathrm{\Omega}} \sum_{a_1, a_2 \geq 0} \lambda^{a_1 - a_2} x_1^{a_1} x_2^{a_2} = \underset{\geq}{\mathrm{\Omega}} \frac{1}{(1 - x_1 \lambda)(1 - \frac{x_2}{\lambda})}.
\end{equation}
To compute this, we need to eliminate $\lambda$ from \eqref{omegact1}. The method for this calculation will be introduced later.

\begin{rem}
In \cite{Andrews12}, the coefficients $a_{i_1\ldots i_k}$ from Definition \ref{dfn-natural} are defined over the field of rational functions in several complex variables over $\mathbb{C}$, with each $\lambda_i$ restricted to a neighborhood of the circle $|\lambda_i|=1$. Furthermore, these coefficients are required to ensure that all series involved are absolutely convergent within their domain of definition.
\end{rem}

Note that the operator $\mathrm{\Omega}_\geq$ is supposed to act only on the Greek letters $\lambda$ or $\mu$, or on the corresponding indexed versions such as $\lambda_i$ and $\mu_i$ in this paper. The parameters unaffected by $\mathrm{\Omega}_\geq$ will be denoted by letters from the Latin alphabet.

An \emph{Elliott rational function} $E$ takes the form
\begin{align}\label{EEcoreMPA}
E = \frac{L}{(1 - M_1(X) \Lambda^{c_1}) (1 - M_2(X) \Lambda^{c_2}) \cdots (1 - M_n(X) \Lambda^{c_n})},
\end{align}
where each $M_i(X)$ is a monomial in the variables $X = \{x_1, \ldots, x_m\}$ that does not involve $\Lambda = \{\lambda_1, \ldots, \lambda_k\}$; $L$ is a Laurent polynomial in $X$ and $\Lambda$. The \emph{core problem} in MacMahon's partition analysis is to compute $\mathrm{\Omega}_\geq E$, which amounts to eliminating the variables $\{\lambda_1, \ldots, \lambda_k\}$ to obtain a rational function in $\{x_1, \ldots, x_m\}$.

MacMahon's partition analysis is particularly suited for solving problems related to counting solutions to linear Diophantine equations and inequalities. Notable applications include enumerating lattice points in convex polytopes and computing Ehrhart quasi-polynomials.

The technique has been revisited by Andrews, Paule, and Riese in a series of papers \cite{Andrews1, Andrews2, Andrews3, Andrews4, Andrews5, Andrews6, Andrews7, Andrews8, Andrews9, Andrews10, Andrews11, Andrews12, Andrews13, Andrews14, Andrews15}, where computer algebra plays a central role. Their work led to the development of an algorithm and corresponding software, such as the \texttt{Omega} package in \texttt{Mathematica} \cite{Andrews3}. The package was later upgraded to \texttt{Omega2} in 2001 \cite{Andrews6}.

In 2003, Han \cite{han} presented a general algorithm for evaluating the MacMahon $\Omega_{\geq}$ operator and developed a \texttt{Maple} package called \texttt{GenOmega}.

In 2017, Breuer and Zafeirakopoulos \cite{PolyhedralOmega} integrated MacMahon's iterative approach based on the Omega operator with explicit formulas for its evaluation and geometric tools such as Brion decompositions and Barvinok's short rational function representations. They introduced a new algorithm for solving linear Diophantine systems, implemented in the \texttt{Sage} package \texttt{Polyhedral Omega}.

In \cite{Xin04}, Xin devised an efficient algorithm for MacMahon's partition analysis by combining the theory of iterated Laurent series with an algorithm for partial fraction decompositions. This algorithm is implemented in the \texttt{Maple} package \texttt{Ell}, which achieves a significantly faster practical running time compared to \texttt{Omega}.
In \cite{Xin15}, Xin proposed a Euclid-style algorithm, with an implementation in the \texttt{Maple} package \texttt{CTEuclid}. This algorithm performs effectively and overcomes several bottlenecks encountered in the \texttt{Ell} package.

In this paper, we introduce a fundamental computational method for evaluating $\mathrm{\Omega}_\geq\ E$. Our approach leverages the key ideas from the partial fraction decomposition algorithm and the rules for series expansion of multivariate rational functions in \cite{Xin04}, and the technique for calculating the contribution of a single factor in \cite{Xin15}. 
The second author previously addressed the problem of extracting constant terms from multivariate rational functions in the same work \cite{Xin15}. Here, we extend certain ideas from that study to the operator $\mathrm{\Omega}_\geq$, offering a parallel but distinct computational framework.

Our method enables straightforward computation of numerous examples from MacMahon's book \cite{MacMahon19} and a series of papers by Andrews and his coauthors \cite{Andrews1, Andrews2, Andrews3, Andrews4, Andrews5, Andrews6, Andrews7, Andrews8, Andrews9, Andrews10, Andrews11, Andrews12, Andrews13, Andrews14, Andrews15}. Previous computational methods mainly rely on the definition of $\mathrm{\Omega}_\geq$ or the \texttt{Omega} package.
As applications, we demonstrate simplified computations for problems such as the ``$k$-gon partitions problem", ``two-dimensional problem", ``a hard problem", and others.

This paper is organized as follows. In Section 2, we present our fundamental computational approach for calculating $\mathrm{\Omega}_\geq\ E$. In Section 3, we illustrate the application of our method through several propositions. Section 4 provides a concise procedure for MacMahon's partition analysis, along with an example to demonstrate the usage of the commands. In Section 5, we reprove two problems with a relatively simple process and solve a hard problem.

\section{Fundamental Calculation Method}

We now present a fundamental computational method for the $\mathrm{\Omega}_\geq$ operator.
In \cite{Xin04}, Xin introduced the field of iterated Laurent series to establish the uniqueness of series expansions for rational functions. Define the working field as $\mathbb{K} = \mathbb{C}((\lambda_k))((\lambda_{k-1}))\cdots((\lambda_1))$.
We adopt terminology from \cite{Xin15}.

Note that every monomial $M \neq 1$ is comparable to $1$ in $\mathbb{K}$ according to the following rule: identify the ``smallest" variable $\lambda_j$ that appears in $M$, i.e., $\deg_{\lambda_i} M = 0$ for all $i < j$. If $\deg_{\lambda_j} M > 0$, then $M$ is said to be \emph{small}, denoted $M < 1$; otherwise, $M$ is \emph{large}, denoted $M > 1$. Consequently, we can determine which of the following series expansions holds in $\mathbb{K}$:
\begin{align*}
\dfrac{1}{1-M} = \begin{cases}
\sum_{i \geq 0} M^i & \text{if } M < 1, \\
\dfrac{1}{-M(1 - 1/M)} = \sum_{i \geq 0} -\dfrac{1}{M^{i+1}} & \text{if } M > 1.
\end{cases}
\end{align*}
Let $\overline{\mathbb{K}}=\mathbb{K}((x_1))\cdots((x_m))$. For Equation \eqref{EEcoreMPA}, when expanding $E$ as a series in $\overline{\mathbb{K}}$, we typically require $E$ to be in its \emph{proper form}, meaning that $M_i(x) \Lambda^{c_i} < 1$ in $\overline{\mathbb{K}}$ for all $i$.

We now employ Xin's partial fraction decomposition algorithm \cite{Xin15} to compute $\mathrm{\Omega}_\geq E$ in $\overline{\mathbb{K}}$ with respect to $\lambda$.
For convenience, we express $E$ in the following form:
\begin{equation}
E(\lambda) = \dfrac{L(\lambda)}{\prod_{i=1}^n (1 - u_i \lambda^{a_i})} \quad \text{(not in proper form)}, \label{NEE}
\end{equation}
where $L(\lambda)$ is a Laurent polynomial, each $u_i$ is a monomial in $\{x_1, \ldots, x_m\} \cup (\Lambda \setminus \{\lambda\})$, and all $a_i$ are \textbf{positive integers}. Note that each $u_i$ is free of $\lambda$.

\begin{thm}\label{propEPFD}
Assume that the partial fraction decomposition of $E(\lambda)$ is given by
\begin{equation}\label{EE}
E(\lambda) = P(\lambda) + \frac{p(\lambda)}{\lambda^s} + \sum_{i=1}^n \frac{A_i(\lambda)}{1 - u_i \lambda^{a_i}},
\end{equation}
where each $u_i$ is independent of $\lambda$; $P(\lambda)$, $p(\lambda)$, and $A_i(\lambda)$ are polynomials, $\deg p(\lambda) < s$, and $\deg A_i(\lambda) < a_i$ for all $i$. Then,
\begin{equation}\label{EE222Omega}
\underset{\geq}{\mathrm{\Omega}} E(\lambda) = P(1) + \sum_{u_i \lambda^{a_i} < 1} \frac{A_i(1)}{1 - u_i}.
\end{equation}
\end{thm}
\begin{proof}
By the definition of $\mathrm{\Omega}_{\geq}$ (Definition \ref{dfn-natural}), if $u_i \lambda^{a_i} < 1$, then
\[
\underset{\geq}{\mathrm{\Omega}} \frac{A_i(\lambda)}{1 - u_i \lambda^{a_i}} = \frac{A_i(1)}{1 - u_i}.
\]
If $u_i \lambda^{a_i} > 1$, we have
\[
\underset{\geq}{\mathrm{\Omega}} \frac{A_i(\lambda)}{1 - u_i \lambda^{a_i}} = \underset{\geq}{\mathrm{\Omega}} \frac{\lambda^{-a_i} A_i(\lambda)}{-u_i \left(1 - \frac{1}{u_i \lambda^{a_i}}\right)} = 0.
\]
The proof is completed by applying Definition \ref{dfn-natural}.
\end{proof}

\begin{thm}{\em \cite{Xin15}}\label{charAS}
Let $E(\lambda)$ be as in \eqref{EE}. Then $A_i(\lambda)$ is uniquely characterized by
\begin{equation}
\left\{
             \begin{array}{lr}
             A_i(\lambda) \equiv E(\lambda) \cdot (1 - u_i \lambda^{a_i}) & \mod \langle 1 - u_i \lambda^{a_i} \rangle, \\
             \deg_{\lambda} A_i < a_i, &
             \end{array}
\right.\label{AS}
\end{equation}
where $\langle 1 - u \lambda^a \rangle$ denotes the ideal generated by $1 - u \lambda^a$.
\end{thm}

The primary objective is to compute $\mathrm{\Omega}_\geq E(\lambda)$ for $\lambda$ as in \eqref{NEE} in $\overline{\mathbb{K}}$.
We aim to compute
\[
\mathcal{T}_{1 - u_i \lambda^{a_i}} E(\lambda) := \frac{A_i(1)}{1 - u_i},
\]
where $A_i(\lambda)$ is characterized by \eqref{AS}. Therefore, in this new notation, Theorem \ref{propEPFD} can be expressed as
\begin{align}\label{EP0}
\underset{\geq}{\mathrm{\Omega}} E(\lambda) = P(1) + \sum_{i} \chi(u_i \lambda^{a_i} < 1) \mathcal{T}_{1 - u_i \lambda^{a_i}} E(\lambda),
\end{align}
where $\chi(\text{true}) = 1$ and $\chi(\text{false}) = 0$.

The following result can be used to avoid computing $P(1)$.
\begin{thm}\label{LemE}
Let $E(\lambda)$ be as in \eqref{NEE}. If $E(\lambda)$ is proper in $\lambda$, i.e., the degree of the numerator is less than that of the denominator, then
\begin{align}\label{CE}
\underset{\geq}{\mathrm{\Omega}}\ E(\lambda)=\sum_{i=1}^n\left(\chi(u_i\lambda^{a_i}<1)\cdot\frac{A_i(1)}{1-u_i}\right).
\end{align}
If $E(0)=\lim_{\lambda\rightarrow 0}E(\lambda)$ and $E(1)$ exists, then
\begin{align}\label{DCE}
\underset{\geq}{\mathrm{\Omega}}\ E(\lambda)=E(1)-\sum_{i=1}^n\left(\chi(u_i\lambda^{a_i}>1)\cdot\frac{A_i(1)}{1-u_i}\right).
\end{align}
\end{thm}
\begin{proof}
Assume the partial fraction decomposition given in \eqref{EE}.
(i) If $E(\lambda)$ is proper in $\lambda$, then $P(\lambda)=0$, and the first equality holds.
(ii) If $E(0)$ exists, then $p(\lambda)$ must be identically zero. Setting $\lambda=1$ and applying \eqref{EP0} yields
\[
E(1)=P(1)+\sum_{i=1}^{n}\frac{A_i(1)}{1-u_i}=\underset{\geq}{\mathrm{\Omega}}\ E(\lambda)+\sum_{i=1}^n\chi(u_i\lambda^{a_i}>1)\mathcal{T}_{1-u_i\lambda^{a_i}}E(\lambda).
\]
The second equality follows by subtracting the sum on the right-hand side.
\end{proof}

Equation \eqref{DCE} can be viewed as a dual form of \eqref{EP0}. Due to these two formulas, it is convenient to refer to the denominator factor $1-u_i\lambda^{a_i}$ as \emph{contributing} when $u_i\lambda^{a_i}$ is small and as \emph{dually contributing} when $u_i\lambda^{a_i}$ is large.

For convenience, we introduce the following notation.
\begin{dfn}\label{UnderlinedContr}
We denote
\begin{align*}
\underset{\geq}{\mathrm{\Omega}}\ \underline{\dfrac{1}{1-u_i\lambda^{a_i}}}E(\lambda)(1-u_i\lambda^{a_i})
=\mathcal{T}_{1-u_i\lambda^{a_i}}E(\lambda)=\frac{A_i(1)}{1-u_i}.
\end{align*}
\end{dfn}
In this notation, when applying $\mathrm{\Omega}_\geq$ in $\lambda$, only the single underlined denominator factor contributes.

Theorems \ref{propEPFD}, \ref{charAS}, and \ref{LemE} establish the fundamental computational framework for the $\mathrm{\Omega}_\geq$ operator. Definition \ref{UnderlinedContr} provides a convenient notational shorthand that streamlines our presentation.
To illustrate the application of our method, we now consider the following example.

\begin{exa}{\em \cite[Page 102]{MacMahon19}}\label{Exam112}
We compute the identity
\begin{align*}
\underset{\geq}{\mathrm{\Omega}}\frac{1}{(1-x\lambda)(1-y\lambda)(1-\frac{z}{\lambda})}=\frac{1-xyz}{(1-x)(1-y)(1-xz)(1-zy)}
\end{align*}
using two distinct approaches. For clarity, we provide detailed step-by-step explanations to facilitate the reader's understanding of our notation.

\textbf{First approach:} We proceed as follows:
\begin{align*}
&\underset{\geq}{\mathrm{\Omega}}\frac{1}{(1-x\lambda)(1-y\lambda)(1-\frac{z}{\lambda})}
=\underset{\geq}{\mathrm{\Omega}}\frac{1}{\underline{(1-x\lambda)}\underline{(1-y\lambda)}(1-\frac{z}{\lambda})} \quad (\text{since both factors contribute}) \\
=&\underset{\geq}{\mathrm{\Omega}}\frac{1}{\underline{(1-x\lambda)}(1-y\lambda)(1-\frac{z}{\lambda})}
+\underset{\geq}{\mathrm{\Omega}}\frac{1}{(1-x\lambda)\underline{(1-y\lambda)}(1-\frac{z}{\lambda})} \\
=&\underset{\geq}{\mathrm{\Omega}}\frac{1}{\underline{(1-x\lambda)}(1-x^{-1}y)(1-xz)}
+\underset{\geq}{\mathrm{\Omega}}\frac{1}{(1-xy^{-1})\underline{(1-y\lambda)}(1-yz)} \quad (\text{by Theorem \ref{charAS}}) \\
=&\frac{1}{(1-x)(1-x^{-1}y)(1-xz)}+\frac{1}{(1-xy^{-1})(1-y)(1-yz)} \quad(\text{setting } \lambda=1 \text{ via \eqref{EE222Omega}}) \\
=&\frac{1-xyz}{(1-x)(1-y)(1-xz)(1-zy)}.
\end{align*}
At the third equality, we apply the operations modulo $\langle 1-x\lambda\rangle$ and modulo $\langle 1-y\lambda\rangle$, respectively, as justified by Theorem \ref{charAS}.

\textbf{Second approach:} By considering the dual contribution, we derive:
\begin{align*}
&\underset{\geq}{\mathrm{\Omega}}\frac{1}{(1-x\lambda)(1-y\lambda)(1-\frac{z}{\lambda})} \\
=&\frac{1}{(1-x)(1-y)(1-z)}-\underset{\geq}{\mathrm{\Omega}}\frac{1}{(1-x\lambda)(1-y\lambda)\underline{(1-\frac{z}{\lambda})}} \quad (\text{by \eqref{DCE}}) \\
=&\frac{1}{(1-x)(1-y)(1-z)}-\underset{\geq}{\mathrm{\Omega}}\frac{-\lambda z^{-1}}{(1-x\lambda)(1-y\lambda)\underline{(1-\frac{\lambda}{z})}} \quad(\text{since } a_i \text{is a positive integer in \eqref{NEE}}) \\
=&\frac{1}{(1-x)(1-y)(1-z)}
-\underset{\geq}{\mathrm{\Omega}}\frac{-1}{(1-xz)(1-yz)\underline{(1-\frac{\lambda}{z})}} \quad (\text{by Theorem \ref{charAS}}) \\
=&\frac{1}{(1-x)(1-y)(1-z)}
-\frac{-1}{(1-xz)(1-yz)(1-\frac{1}{z})} \quad( \text{setting } \lambda=1 \text{ via \eqref{DCE}}) \\
=&\frac{1-xyz}{(1-x)(1-y)(1-xz)(1-zy)}.
\end{align*}
In fact, the second approach can be simplified as follows:
\begin{align*}
&\underset{\geq}{\mathrm{\Omega}}\frac{1}{(1-x\lambda)(1-y\lambda)(1-\frac{z}{\lambda})}
=\frac{1}{(1-x)(1-y)(1-z)}-\underset{\geq}{\mathrm{\Omega}}\frac{1}{(1-x\lambda)(1-y\lambda)\underline{(1-\frac{z}{\lambda})}} \\
=&\frac{1}{(1-x)(1-y)(1-z)}-\underset{\geq}{\mathrm{\Omega}}\frac{\lambda}{(1-x\lambda)(1-y\lambda)\underline{(\lambda-z)}} \\
=&\frac{1}{(1-x)(1-y)(1-z)}-\underset{\geq}{\mathrm{\Omega}}\frac{z}{(1-xz)(1-yz)\underline{(\lambda-z)}} \\
=&\frac{1}{(1-x)(1-y)(1-z)}-\frac{z}{(1-xz)(1-yz)(1-z)} \\
=&\frac{1-xyz}{(1-x)(1-y)(1-xz)(1-zy)}.
\end{align*}
\end{exa}

\section{Fundamental Formulas}

\begin{prop}\label{omegaprop}
Let $A\neq 1$ be independent of $\lambda$. If $F(\lambda^{-1})$ is a rational function in $\lambda^{-1}$ that admits a formal power series expansion in $\lambda^{-1}$, then
\begin{align*}
\underset{\geq}{\mathrm{\Omega}}\ \frac{1}{1 - A\lambda} \cdot F(\lambda^{-1}) = \frac{1}{1 - A} \cdot F(A).
\end{align*}
\end{prop}
\begin{proof}
Method 1: For any positive integer $m$, by Definition \ref{dfn-natural} and the geometric series expansion, we have
\begin{align*}
\underset{\geq}{\mathrm{\Omega}}\ \frac{1}{1-A\lambda}\cdot \lambda^{-m}
=\underset{\geq}{\mathrm{\Omega}}\sum_{k\geq 0}\lambda^kA^k\cdot \lambda^{-m}
=\underset{\geq}{\mathrm{\Omega}}\sum_{k\geq 0}\lambda^{k-m}A^k=\sum_{k\geq m}A^k=\frac{A^m}{1-A}.
\end{align*}
The result follows from the linearity of the $\mathrm{\Omega}_{\geq}$ operator.

Method 2: Using the underlined notation for the contributing factor, we obtain
\begin{align*}
\underset{\geq}{\mathrm{\Omega}}\ \frac{1}{1-A\lambda}\cdot F(\lambda ^{-1})=\underset{\geq}{\mathrm{\Omega}}\ \underline{\frac{1}{1-A\lambda}}\cdot F(\lambda ^{-1})=\frac{1}{1-A}\cdot F(A),
\end{align*}
which completes the proof.
\end{proof}

\begin{prop}\label{omegauLamb}
Let $A \neq 1$ be independent of $\lambda$. If $F(\lambda)$ is a rational function in $\lambda$ that admits a formal power series expansion in $\lambda$, then
\begin{align*}
\underset{\geq}{\mathrm{\Omega}}\ \frac{1}{1 - \frac{A}{\lambda}} \cdot F(\lambda) = \frac{F(1) - A \cdot F(A)}{1 - A}.
\end{align*}
\end{prop}
\begin{proof}
We proceed as follows:
\begin{align*}
\underset{\geq}{\mathrm{\Omega}}\ \frac{1}{1-\frac{A}{\lambda}}\cdot F(\lambda)
&=\frac{1}{1-A}\cdot F(1)-\underset{\geq}{\mathrm{\Omega}}\ \underline{\frac{1}{1-\frac{A}{\lambda}}}\cdot F(\lambda)
=\frac{1}{1-A}\cdot F(1)-\underset{\geq}{\mathrm{\Omega}}\ \underline{\frac{\lambda}{\lambda-A}}\cdot F(\lambda)
\\&=\frac{1}{1-A}\cdot F(1)-\underset{\geq}{\mathrm{\Omega}}\ \underline{\frac{A}{\lambda-A}}\cdot F(A)
=\frac{1}{1-A}\cdot F(1)-\frac{A}{1-A}\cdot F(A).
\end{align*}
Here, the second step rewrites the denominator, and the third step applies the operator to the underlined factor. The final expression yields the desired result.
\end{proof}

The following nine fundamental formulas are presented in the book \cite[Pages 102--103]{MacMahon19} without proof. Using our method, these identities can be readily established.
\begin{prop}{\em \cite[Pages 102--103]{MacMahon19} and \cite[Section 2]{Andrews3}}
Let $s$ be a positive integer. The fundamental evaluations of the Omega operator are given by:
\begin{align}
\underset{\geq}{\mathrm{\Omega}}\ \frac{1}{(1-\lambda x)(1-\frac{y}{\lambda^s})}&=\frac{1}{(1-x)(1-x^sy)},\label{fundamental-1}
\\ \underset{\geq}{\mathrm{\Omega}}\ \frac{1}{(1-\lambda^sx)(1-\frac{y}{\lambda})}&=\frac{1+xy\frac{1-y^{s-1}}{1-y}}{(1-x)(1-y^sx)},\label{fundamental-2}
\\ \underset{\geq}{\mathrm{\Omega}}\ \frac{1}{(1-\lambda x)(1-\frac{y}{\lambda})(1-\frac{z}{\lambda})}&=\frac{1}{(1-x)(1-xy)(1-xz)},\label{fundamental-3}
\\ \underset{\geq}{\mathrm{\Omega}}\ \frac{1}{(1-\lambda x)(1-\lambda y)(1-\frac{z}{\lambda})}&=\frac{1-xyz}{(1-x)(1-y)(1-xz)(1-zy)},\label{fundamental-4}
\\ \underset{\geq}{\mathrm{\Omega}}\ \frac{1}{(1-\lambda x)(1-\lambda y)(1-\frac{z}{\lambda^2})}&=\frac{1+xyz-x^2yz-xy^2z}{(1-x)(1-y)(1-zx^2)(1-zy^2)},\label{fundamental-5}
\\ \underset{\geq}{\mathrm{\Omega}}\ \frac{1}{(1-\lambda^2 x)(1-\frac{y}{\lambda})(1-\frac{z}{\lambda})}&=\frac{1+xy+xz+xyz}{(1-x)(1-xy^2)(1-xz^2)},\label{fundamental-6}
\\ \underset{\geq}{\mathrm{\Omega}}\ \frac{1}{(1-\lambda^2 x)(1-\lambda y)(1-\frac{z}{\lambda})}&=\frac{1+xz-xyz-xyz^2}{(1-x)(1-y)(1-yz)(1-xz^2)},\label{fundamental-7}
\\ \underset{\geq}{\mathrm{\Omega}}\ \frac{1}{(1-\lambda x)(1-\lambda y)(1-\lambda z)(1-\frac{w}{\lambda})}&=\frac{1-xyw-xzw-yzw+xyzw+xyzw^2}{(1-x)(1-y)(1-z)(1-wx)(1-wy)(1-wz)},\label{fundamental-8}
\\ \underset{\geq}{\mathrm{\Omega}}\ \frac{1}{(1-\lambda x)(1-\lambda y)(1-\frac{z}{\lambda})(1-\frac{w}{\lambda})}&=\frac{1-xyz-xyw-xyzw+xy^2zw+x^2yzw}{(1-x)(1-y)(1-xz)(1-xw)(1-yz)(1-yw)}.\label{fundamental-9}
\end{align}
\end{prop}
\begin{proof}
Equations \eqref{fundamental-1} and \eqref{fundamental-3} are direct consequences of Proposition \ref{omegaprop}. Equation \eqref{fundamental-2} follows from Proposition \ref{omegauLamb}. The identity \eqref{fundamental-4} is verified in Example \ref{Exam112}.
The proofs of \eqref{fundamental-5} and \eqref{fundamental-9} are analogous to the first approach in Example \ref{Exam112}.
Similarly, \eqref{fundamental-7} and \eqref{fundamental-8} can be derived using the second method described in Example \ref{Exam112}.

For \eqref{fundamental-6}, we proceed as follows:
\begin{align*}
&\underset{\geq}{\mathrm{\Omega}}\ \frac{1}{(1-\lambda^2 x)(1-\frac{y}{\lambda})(1-\frac{z}{\lambda})}
=\underset{\geq}{\mathrm{\Omega}}\ \frac{1}{\underline{(1-\lambda^2 x)}(1-\frac{y}{\lambda})(1-\frac{z}{\lambda})}
=\underset{\geq}{\mathrm{\Omega}}\ \frac{(1+\frac{y}{\lambda})(1+\frac{z}{\lambda})}{\underline{(1-\lambda^2 x)}(1-\frac{y^2}{\lambda^2})(1-\frac{z^2}{\lambda^2})}
\\=& \underset{\geq}{\mathrm{\Omega}}\ \frac{1+\frac{y}{\lambda}\cdot \lambda^2x+\frac{z}{\lambda}\cdot \lambda^2x+\frac{yz}{\lambda^2}\cdot \lambda^2x}{\underline{(1-\lambda^2 x)}(1-xy^2)(1-xz^2)}
=\frac{1+xy+xz+xyz}{(1-x)(1-xy^2)(1-xz^2)}.
\end{align*}
This completes the proof.
\end{proof}

As noted in \cite{Andrews5}:
\emph{``The proofs of many these rules are quite elementary but in case of several $\lambda$ variables, elimination can be much more cumbersome."}
We now demonstrate through the following example that our method remains straightforward even in the presence of multiple $\lambda$ variables.

\begin{prop}{\em \cite[Lemma 2.1]{Andrews11}}
The following identity holds:
\begin{align*}
&\underset{\geq}{\mathrm{\Omega}}\ \frac{1-AB\lambda_1\lambda_2}{(1-A\lambda_1)(1-B\lambda_2)(1-C\lambda_1\lambda_2)(1-D\lambda_1\lambda_2)(1-\frac{E}{\lambda_1\lambda_2})}
\\=&\frac{(1-AB)(1-CDE)}{(1-A)(1-B)(1-C)(1-D)(1-CE)(1-DE)}.
\end{align*}
\end{prop}
\begin{proof}
The operator $\mathrm{\Omega}_{\geq}$ acts to eliminate the variables $\lambda_1$ and $\lambda_2$.
We proceed by first eliminating $\lambda_{1}$, followed by $\lambda_{2}$.
By accounting for the dual contribution in $\lambda_1$, we derive
\begin{footnotesize}
\begin{align*}
&\underset{\geq}{\mathrm{\Omega}}\ \frac{1-AB\lambda_1\lambda_2}{(1-A\lambda_1)(1-B\lambda_2)(1-C\lambda_1\lambda_2)(1-D\lambda_1\lambda_2)
(1-\frac{E}{\lambda_1\lambda_2})}
\\=&\underset{\geq}{\mathrm{\Omega}}\ \frac{1-AB\lambda_2}{(1-A)(1-B\lambda_2)(1-C\lambda_2)(1-D\lambda_2)(1-\frac{E}{\lambda_2})}
-\underset{\geq}{\mathrm{\Omega}}\ \frac{1-AB\lambda_1\lambda_2}{(1-A\lambda_1)(1-B\lambda_2)(1-C\lambda_1\lambda_2)(1-D\lambda_1\lambda_2)
\underline{(1-\frac{E}{\lambda_1\lambda_2})}}
\\=&\underset{\geq}{\mathrm{\Omega}}\ \frac{1-AB\lambda_2}{(1-A)(1-B\lambda_2)(1-C\lambda_2)(1-D\lambda_2)(1-\frac{E}{\lambda_2})}
-\underset{\geq}{\mathrm{\Omega}}\ \frac{(1-ABE)\cdot\frac{E}{\lambda_2}}{(1-\frac{AE}{\lambda_2})(1-B\lambda_2)(1-CE)(1-DE)(1-\frac{E}{\lambda_2})}.
\end{align*}
\end{footnotesize}
For the first term above, we consider the dual contribution in $\lambda_2$, while for the second term, we account for the contribution in $\lambda_2$. This yields
\begin{footnotesize}
\begin{align*}
&\underset{\geq}{\mathrm{\Omega}}\ \frac{1-AB\lambda_1\lambda_2}{(1-A\lambda_1)(1-B\lambda_2)(1-C\lambda_1\lambda_2)(1-D\lambda_1\lambda_2)
(1-\frac{E}{\lambda_1\lambda_2})}
\\=&\frac{1-AB}{(1-A)(1-B)(1-C)(1-D)(1-E)}-\underset{\geq}{\mathrm{\Omega}}\ \frac{1-AB\lambda_2}{(1-A)(1-B\lambda_2)(1-C\lambda_2)(1-D\lambda_2)\underline{(1-\frac{E}{\lambda_2})}}
\\&\ \ \ \ \ \ \ \ \ \ -\underset{\geq}{\mathrm{\Omega}}\ \frac{(1-ABE)\cdot\frac{E}{\lambda_2}}{(1-\frac{AE}{\lambda_2})\underline{(1-B\lambda_2)}(1-CE)(1-DE)(1-\frac{E}{\lambda_2})}
\\=&\frac{1-AB}{(1-A)(1-B)(1-C)(1-D)(1-E)}-\frac{(1-ABE)E}{(1-A)(1-BE)(1-CE)(1-DE)(1-E)}
\\&\ \ \ \ \ \ \ \ \ \  -\frac{(1-ABE)BE}{(1-ABE)(1-B)(1-CE)(1-DE)(1-BE)}
\\=&\frac{(1-AB)(1-CDE)}{(1-A)(1-B)(1-C)(1-D)(1-CE)(1-DE)}.
\end{align*}
\end{footnotesize}
This concludes the proof.
\end{proof}

We have also computed numerous examples from the works of Andrews and his collaborators.
For instance: \cite[Equation (12)]{Andrews5}, \cite[Equation (3.1)]{Andrews1}, \cite[Equations (3) and (5)]{Andrews9}, \cite[Lemmas 1 and 2]{Andrews9}, \cite[Proposition 1]{Andrews10}, \cite[Lemma 4.1]{Andrews14}, \cite[Lemmas 2.2, 2.3, 2.4, and 2.5]{Andrews15}, among others. Our approach readily handles these computations.

\section{A Computational Framework}

The following result first appeared in \cite[Page 183]{MacMahon19}, and has been frequently utilized in the work of Andrews et al. \cite{Andrews8, Andrews10, Andrews11, Andrews12}.
\begin{prop}{\em \cite[Page 183]{MacMahon19}}\label{MacMahon19-Page183}
The following identity holds:
\begin{align*}
&\underset{\geq}{\mathrm{\Omega}}\ \frac{1}{(1-x_1\lambda_1\lambda_2)(1-\frac{x_2\lambda_3}{\lambda_1})
(1-\frac{x_4}{\lambda_3\lambda_4})(1-\frac{x_3\lambda_4}{\lambda_2})}
\\=&\frac{1-x_1^2x_2x_3}{(1-x_1)(1-x_1x_2)(1-x_1x_3)(1-x_1x_2x_3)(1-x_1x_2x_3x_4)}.
\end{align*}
\end{prop}
\begin{proof}
The operator $\mathrm{\Omega}_{\geq}$ is applied to eliminate the variables $\lambda_1, \lambda_2, \lambda_3, \lambda_4$. We proceed by eliminating them in the order $\lambda_4, \lambda_3, \lambda_2, \lambda_1$ successively.
This yields the following sequence of expressions:
\begin{align*}
&\underset{\geq}{\mathrm{\Omega}}\ \frac{1}{(1-x_1\lambda_1\lambda_2)(1-\frac{x_2\lambda_3}{\lambda_1})
(1-\frac{x_4}{\lambda_3\lambda_4})\underline{(1-\frac{x_3\lambda_4}{\lambda_2})}}
=\underset{\geq}{\mathrm{\Omega}}\ \frac{1}{(1-x_1\lambda_1\lambda_2)\underline{(1-\frac{x_2\lambda_3}{\lambda_1})}
(1-\frac{x_3x_4}{\lambda_2\lambda_3})(1-\frac{x_3}{\lambda_2})}
\\=&\underset{\geq}{\mathrm{\Omega}}\ \frac{1}{\underline{(1-x_1\lambda_1\lambda_2)}(1-\frac{x_2}{\lambda_1})
(1-\frac{x_2x_3x_4}{\lambda_1\lambda_2})(1-\frac{x_3}{\lambda_2})}
=\underset{\geq}{\mathrm{\Omega}}\ \frac{1}{\underline{(1-x_1\lambda_1)}(1-\frac{x_2}{\lambda_1})
(1-x_1x_2x_3x_4)\underline{(1-x_1x_3\lambda_1)}}
\\=&\frac{1}{(1-x_1)(1-x_1x_2)(1-x_3)(1-x_1x_2x_3x_4)}+\frac{1}{(1-x_3^{-1})(1-x_1x_2x_3)(1-x_1x_3)(1-x_1x_2x_3x_4)}
\\=&\frac{1-x_1^2x_2x_3}{(1-x_1)(1-x_1x_2)(1-x_1x_3)(1-x_1x_2x_3)(1-x_1x_2x_3x_4)}.
\end{align*}
This completes the proof.
\end{proof}

For the reader's convenience, we have implemented several computational steps in \texttt{Maple} \cite{Maple}, which is freely available at \url{https://github.com/TygerLiu/TygerLiu.github.io/tree/main/Procedure/Omega-Fundamental}.
The program features five main commands: \texttt{R2List}, \texttt{L2Rational}, \texttt{ClassifyNCDI}, \texttt{ContributeDispel}, and \texttt{DualContributeDispel}.
To illustrate their usage, we now apply these commands to Proposition \ref{MacMahon19-Page183}.

\textbf{Step 1}: We begin by using the command \texttt{ClassifyNCDI} to generate a table.
This table displays the ``contribution terms" and ``dual contribution terms" for each variable $\lambda_i$.
The column ``C-Num" indicates the number of contribution terms, while ``DC-Num" indicates the number of dual contribution terms.
\begin{flalign*}
&>~\text{temp}:=\frac{1}{(1-x_1\lambda_1\lambda_2)(1-\frac{x_2\lambda_3}{\lambda_1})
(1-\frac{x_3\lambda_4}{\lambda_2})(1-\frac{x_4}{\lambda_3\lambda_4})}:\\
&>~\texttt{ClassifyNCDI}(\text{temp});\\
&>~\begin{array}{|c|c|c|c|c|c|}\hline
  & \text{Numerator} & \text{Contributing} & \text{C-Num} & \text{Dully contributing} & \text{DC-Num} \\ \hline
\text{lambda[1]} & 1 & 1-x_1\lambda_1\lambda_2 & 1 & 1-\frac{x_2\lambda_3}{\lambda_1} & 1 \\ \hline
\text{lambda[2]} & 1 & 1-x_1\lambda_1\lambda_2 & 1 & 1-\frac{x_3\lambda_4}{\lambda_2} & 1 \\ \hline
\text{lambda[3]} & 1 & 1-\frac{x_2\lambda_3}{\lambda_1} & 1 & 1-\frac{x_4}{\lambda_3\lambda_4} & 1 \\ \hline
\text{lambda[4]} & 1 & 1-\frac{x_3\lambda_4}{\lambda_2} & 1 & 1-\frac{x_4}{\lambda_3\lambda_4} & 1 \\ \hline
\end{array}&
\end{flalign*}

\textbf{Step 2}: Next, we apply the command \texttt{R2List} to convert the rational function ``temp" into an equivalent list representation.
The resulting list ``LL" provides a form of ``temp" that facilitates computation in \texttt{Maple}.
The first element of the list corresponds to the numerator of the rational function.
\begin{flalign*}
&>~\texttt{R2List}(\text{temp},[x_1,x_2,x_3,x_4,\lambda_1,\lambda_2,\lambda_3,\lambda_4]);\\
&>~\left[-\frac{1}{x_1x_2x_3\lambda_3\lambda_4},\frac{1}{x_1\lambda_1\lambda_2},\frac{\lambda_1}{x_2\lambda_3},
\frac{\lambda_2}{x_3\lambda_4},\frac{x_4}{\lambda_3\lambda_4}\right]\\
&>~\text{LL}:=\texttt{Lflip}(\%,(M)\rightarrow \text{not}\ \  \texttt{IsSmall}(M,[x_1,x_2,x_3,x_4,\lambda_1,\lambda_2,\lambda_3,\lambda_4]));\\
&>~\text{LL}:=\left[1,x_1\lambda_1\lambda_2,\frac{x_2\lambda_3}{\lambda_1},\frac{x_3\lambda_4}{\lambda_2},
\frac{x_4}{\lambda_3\lambda_4}\right]&
\end{flalign*}

\textbf{Step 3}: Based on the table above, we first eliminate the variable $\lambda_4$, as it appears in only one contribution term.
This is achieved using the command \texttt{ContributeDispel}, which corresponds to the underlined operation.
For example, $\texttt{ContributeDispel}(LL,\lambda_4,3)$ specifies that the $(3+1)$-th term of $\lambda_4$ in ``LL" is the contribution term.
We then iteratively apply \texttt{ContributeDispel} to remove the remaining variables.
The command \texttt{L2Rational} converts the list back into a rational function.
\begin{flalign*}
&>~\text{T1}:=\texttt{ContributeDispel}(\text{LL},\lambda_4,3);\\
&>~T1:=\left[1,-\frac{x_3x_4}{\lambda_2\lambda_3},x_1\lambda_1\lambda_2,\frac{x_2\lambda_3}{\lambda_1},
\frac{x_3}{\lambda_2}\right]\\
&>~\text{T2}:=\texttt{ContributeDispel}(\text{T1},\lambda_3,3);\\
&>~T2:=\left[1,-\frac{x_2x_3x_4}{\lambda_1\lambda_2},x_1\lambda_1\lambda_2,\frac{x_3}{\lambda_2},
\frac{x_2}{\lambda_1}\right]\\
&>~\text{T3}:=\texttt{ContributeDispel}(\text{T2},\lambda_2,2); \text{F3}:=\texttt{L2Rational}(T3);\\
&>~T3:=\left[1, x_1x_2x_3x_4, x_1x_3\lambda_1, \frac{x_2}{\lambda_1}, x_1\lambda_1\right]\ F3:=\frac{\lambda_1}{(1-x_1x_2x_3x_4)(1-x_1x_3\lambda_1)(\lambda_1-x_2)(1-x_1\lambda_1)}\\
&>~\text{T41}:=\texttt{ContributeDispel}(\text{T3},\lambda_1,2); \text{F41}:=\texttt{L2Rational}(T41);\\
&>~T41:=\left[1, x_1x_2x_3, \frac{1}{x_3}, x_1x_2x_3x_4, x_1x_3\right]\ F41:=\frac{x_3}{(1-x_1x_2x_3)(1-x_1x_3)(1-x_3)(1-x_1x_2x_3x_4)}\\
&>~\text{T42}:=\texttt{ContributeDispel}(\text{T3},\lambda_1,4); \text{F42}:=\texttt{L2Rational}(T42);\\
&>~T42:=\left[1, x_3, x_1x_2, x_1x_2x_3x_4, x_1\right]\ F42:=\frac{1}{(1-x_3)(1-x_1x_2)(1-x_1)(1-x_1x_2x_3x_4)}\\
&>~\text{F}:=\texttt{normal}(F41+F42);\\
&>~F:=\frac{1-x_1^2x_2x_3}{(1-x_1)(1-x_1x_2)(1-x_1x_3)(1-x_1x_2x_3x_4)(1-x_1x_2x_3)}
\end{flalign*}

\textbf{Step 4}: When we consider the dual contribution, we employ the command \texttt{DualContributeDispel}.
\begin{flalign*}
&>~\text{T5}:=\texttt{DualContributeDispel}(\text{T3},\lambda_1,3); \text{F5}:=\texttt{L2Rational}(T5);\\
&>~T5:=\left[x_2, x_1x_2x_3, x_1x_2, x_1x_2x_3x_4, x_2\right]\ F5:=\frac{x_2}{(1-x_1x_2x_3)(1-x_1x_2)(1-x_2)(1-x_1x_2x_3x_4)}\\
&>~\text{F}:=\texttt{normal}(\texttt{subs}(\lambda_1=1,F3)-F5);\\
&>~F:=\frac{1-x_1^2x_2x_3}{(1-x_1)(1-x_1x_2)(1-x_1x_3)(1-x_1x_2x_3x_4)(1-x_1x_2x_3)}
\end{flalign*}

\section{Applications}

\subsection{Han's Formula}

Han \cite{han} developed a general algorithm for computing the MacMahon $\Omega_{\geq}$ operator and implemented it in the \texttt{Maple} package \texttt{GenOmega}. Equation \eqref{Han-Formula} below is a key result of this algorithm, and its proof occupies a significant portion \cite[Theorem 1.4]{han}. We now need only one step to prove this formula.

Let
$$A(t)=\prod_{i=1}^n(1-x_it),\quad B(t)=\prod_{i=1}^m(1-y_it)$$
be two polynomials.
\begin{thm}{\em \cite[Theorem 1.4]{han}}
Let $U(\lambda)$ be a Laurent polynomial with $x_i\neq 1$ and $x_iy_j\neq 1$ for all $i,j$. If $\deg U(\lambda)\leq n-1$, then
\begin{align}\label{Han-Formula}
\underset{\geq}{\mathrm{\Omega}}\ \frac{U(\lambda)}{A(\lambda)B(1/\lambda)}=\sum_{i=1}^n\frac{x_i^{n-1}U(1/x_i)}{(1-x_i)B(x_i)\prod_{j\neq i}(x_i-x_j)}.
\end{align}
The formula holds even when the $y_i$'s are not all distinct.
\end{thm}
\begin{proof}
We compute as follows:
\begin{align*}
\underset{\geq}{\mathrm{\Omega}}\ \frac{U(\lambda)}{A(\lambda)B(1/\lambda)}&=\underset{\geq}{\mathrm{\Omega}}\ \frac{U(\lambda)}{\underline{(1-x_1\lambda)(1-x_2\lambda)\cdots (1-x_n\lambda)}B(1/\lambda)}
\\&=\sum_{i=1}^n \frac{U(1/x_i)}{(1-x_i)B(x_i)\prod_{j\neq i}(1-x_j/x_i)}.
\end{align*}
This completes the proof.
\end{proof}

\subsection{$k$-gon Partitions}

For simplicity, we adopt the notation from \cite{Andrews9}. Let $\mathbb{P}$ denote the set of all positive integers.
The set of \emph{non-degenerate $k$-gon partitions} into positive parts is defined by
\[
\tau_k = \left\{ (a_1, a_2, \ldots, a_k) \in \mathbb{P}^k \mid 1 \leq a_1 \leq a_2 \leq \cdots \leq a_k,\ a_1 + \cdots + a_{k-1} > a_k \right\}.
\]
The set of \emph{non-degenerate $k$-gon partitions of a positive integer $n$} is given by
\[
\tau_k(n) = \left\{ (a_1, a_2, \ldots, a_k) \in \tau_k \mid a_1 + a_2 + \cdots + a_k = n \right\},
\]
and its cardinality is denoted by
\[
t_k(n) = |\tau_k(n)|.
\]
For any integer $k \geq 3$, we define the generating functions
\[
T_k(q) = \sum_{n \geq k} t_k(n) q^n,
\]
and
\[
S_k(x_1, \ldots, x_k) = \sum_{(a_1, \ldots, a_k) \in \tau_k} x_1^{a_1} x_2^{a_2} \cdots x_k^{a_k}.
\]

In \cite{Andrews3}, Andrews, Paule, and Riese derived explicit expressions for $T_3(q)$, $T_4(q)$, $T_5(q)$, and $T_6(q)$ using the \texttt{Omega} package. They also posed the following open problem \cite[Problem 4.2]{Andrews3}: Is it possible to identify a unified pattern for all $k$? This question was later resolved in \cite{Andrews9}.
In what follows, we present a simplified proof of this result.

\begin{thm}{\em \cite[Theorem 1]{Andrews9}}
Let $k\geq 3$. Then
\begin{small}
\begin{align*}
S_k(x_1,&\ldots,x_k)=\frac{x_1x_2\cdots x_k}{(1-x_1x_2\cdots x_k)(1-x_2\cdots x_k)\cdots(1-x_k)}
\\&-\frac{x_1x_2\cdots x_{k-1}x_k^{k-1}}{(1-x_1\cdots x_{k-1}x_k^{k-1})(1-x_2\cdots x_{k-1}x_k^{k-2})\cdots (1-x_{k-2}x_{k-1}x_k^2)(1-x_{k-1}x_k)(1-x_k)}.
\end{align*}
\end{small}
\end{thm}
\begin{proof}
By the definition of $\mathrm{\Omega}_\geq$ and the geometric series summation, we have
\begin{align*}
S_k(x_1,\ldots,x_k)&=\underset{\geq}{\mathrm{\Omega}}\ \sum_{a_1\geq 1,\ a_2,\ldots,a_k\geq 0}x_1^{a_1}\cdots x_k^{a_k}\lambda_1^{a_2-a_1}\cdots \lambda_{k-1}^{a_k-a_{k-1}}\lambda_{k}^{a_1+\cdots+a_{k-1}-a_k-1}
\\&=\underset{\geq}{\mathrm{\Omega}}\ \frac{x_1\lambda_1^{-1}}{(1-\frac{x_1\lambda_k}{\lambda_1})(1-\frac{x_2\lambda_1\lambda_k}{\lambda_2})
(1-\frac{x_3\lambda_2\lambda_k}{\lambda_3})\cdots(1-\frac{x_{k-1}\lambda_{k-2}\lambda_k}{\lambda_{k-1}})
(1-\frac{x_k\lambda_{k-1}}{\lambda_k})}.
\end{align*}
The $\mathrm{\Omega}_{\geq}$ operator eliminates the variables $\lambda_1,\lambda_2,\ldots,\lambda_k$. We proceed by first eliminating $\lambda_{k-1},\lambda_{k-2},\ldots,\lambda_{1}$ in succession, and then $\lambda_k$.
The computation unfolds as follows:
\begin{align*}
&S_k(x_1,\ldots,x_k)=\underset{\geq}{\mathrm{\Omega}}\ \frac{x_1\lambda_1^{-1}}{(1-\frac{x_1\lambda_k}{\lambda_1})(1-\frac{x_2\lambda_1\lambda_k}{\lambda_2})
(1-\frac{x_3\lambda_2\lambda_k}{\lambda_3})\cdots(1-\frac{x_{k-1}\lambda_{k-2}\lambda_k}{\lambda_{k-1}})
\underline{(1-\frac{x_k\lambda_{k-1}}{\lambda_k})}}
\\=&\underset{\geq}{\mathrm{\Omega}}\ \frac{x_1\lambda_1^{-1}}{(1-\frac{x_1\lambda_k}{\lambda_1})(1-\frac{x_2\lambda_1\lambda_k}{\lambda_2})
\cdots(1-\frac{x_{k-2}\lambda_{k-3}\lambda_k}{\lambda_{k-2}})\underline{(1-x_{k-1}x_k\lambda_{k-2})}
(1-\frac{x_k}{\lambda_k})}
\\=&\underset{\geq}{\mathrm{\Omega}}\ \frac{x_1\lambda_1^{-1}}{(1-\frac{x_1\lambda_k}{\lambda_1})(1-\frac{x_2\lambda_1\lambda_k}{\lambda_2})
\cdots(1-\frac{x_{k-3}\lambda_{k-4}\lambda_k}{\lambda_{k-3}})\underline{(1-x_{k-2}x_{k-1}x_k\lambda_{k-3}\lambda_k)}(1-x_{k-1}x_k)
(1-\frac{x_k}{\lambda_k})}
\\=&\underset{\geq}{\mathrm{\Omega}}\ \frac{x_1\lambda_1^{-1}}{(1-\frac{x_1\lambda_k}{\lambda_1})(1-\frac{x_2\lambda_1\lambda_k}{\lambda_2})
\cdots\underline{(1-x_{k-3}\cdots x_k\lambda_{k-4}\lambda_k^2)}(1-x_{k-2}x_{k-1}x_k\lambda_k)(1-x_{k-1}x_k)
(1-\frac{x_k}{\lambda_k})}
\\=&\cdots
\\=&\underset{\geq}{\mathrm{\Omega}}\ \frac{x_1\lambda_1^{-1}}{(1-\frac{x_1\lambda_k}{\lambda_1})\underline{(1-x_2\cdots x_k\lambda_1\lambda_k^{k-3})}(1-x_3\cdots x_k\lambda_k^{k-4})\cdots (1-x_{k-2}x_{k-1}x_k\lambda_k)(1-x_{k-1}x_k)(1-\frac{x_k}{\lambda_k})}
\\=&\underset{\geq}{\mathrm{\Omega}}\ \frac{x_1x_2\cdots x_k\lambda_k^{k-3}}{(1-x_1\cdots x_k\lambda_k^{k-2})(1-x_2\cdots x_k\lambda_k^{k-3})\cdots (1-x_{k-2}x_{k-1}x_k\lambda_k)(1-x_{k-1}x_k)(1-\frac{x_k}{\lambda_k})}.
\end{align*}
Now, by considering the dual contribution, we obtain
\begin{align*}
S_k(x_1,&\ldots,x_k)=\frac{x_1x_2\cdots x_k}{(1-x_1x_2\cdots x_k)(1-x_2\cdots x_k)\cdots(1-x_k)}
\\&-\underset{\geq}{\mathrm{\Omega}}\ \frac{x_1x_2\cdots x_k\lambda_k^{k-3}}{(1-x_1\cdots x_k\lambda_k^{k-2})(1-x_2\cdots x_k\lambda_k^{k-3})\cdots (1-x_{k-2}x_{k-1}x_k\lambda_k)(1-x_{k-1}x_k)\underline{(1-\frac{x_k}{\lambda_k})}}.
\end{align*}
An application of \eqref{DCE} completes the proof.
\end{proof}

Let $x_i=q$ for all $i$. We obtain the following result.
\begin{cor}{\em \cite[Corollary 1]{Andrews9}}
Let $k\geq 3$. Then
\begin{align*}
T_k(q)=\frac{q^k}{(1-q)(1-q^2)\cdots (1-q^k)}-\frac{q^{2k-2}}{1-q}\cdot \frac{1}{(1-q^2)(1-q^4)\cdots (1-q^{2k-2})}.
\end{align*}
\end{cor}

\subsection{Two Dimensional Problem}

Inspired by the work of \cite{Klosinski}, Andrews, Paule, and Riese investigated the generating function for a two-dimensional problem:
\[
S_{K,L}(x,y)=\sum_{\substack{m,n\geq 0;\\ Km\geq n;\\ Ln\geq m}} x^m y^n,
\]
where $K,L\geq 2$. We now illustrate the application of our method to this problem.

\begin{thm}{\em \cite[Section 2]{Andrews7}}
Let $K,L\geq 2$. Then
\begin{align*}
S_{K,L}(x,y)=\frac{1+xy\frac{(1-x^L)(1-y^K)}{(1-x)(1-y)}-xy^K-x^Ly}{(1-xy^K)(1-x^Ly)}.
\end{align*}
\end{thm}
\begin{proof}
By the definition of the $\mathrm{\Omega}_\geq$ operator and the geometric series summation, we have
\[
S_{K,L}(x,y) = \underset{\geq}{\mathrm{\Omega}} \frac{1}{(1 - \frac{x\lambda_1^K}{\lambda_2})(1 - \frac{y\lambda_2^L}{\lambda_1})}.
\]
To apply the $\mathrm{\Omega}_{\geq}$ operator, we eliminate $\lambda_1$ first, followed by $\lambda_2$.
Considering the dual contribution in $\lambda_1$, we obtain
\begin{align*}
S_{K,L}(x,y)&=\underset{\geq}{\mathrm{\Omega}}\ \frac{1}{(1-\frac{x}{\lambda_2})(1-y\lambda_2^L)}-\underset{\geq}{\mathrm{\Omega}}\ \frac{1}{(1-\frac{x\lambda_1^K}{\lambda_2})\underline{(1-\frac{y\lambda_2^L}{\lambda_1})}}
\\&=\underset{\geq}{\mathrm{\Omega}}\ \frac{1}{(1-\frac{x}{\lambda_2})(1-y\lambda_2^L)}-\underset{\geq}{\mathrm{\Omega}}\ \frac{y\lambda_2^{L}}{(1-xy^K\lambda_2^{LK-1})(1-y\lambda_2^L)}.
\end{align*}
The second term can be evaluated directly by setting $\lambda_2 = 1$. For the first term, we consider the dual contribution in $\lambda_2$, yielding
\begin{align*}
S_{K,L}(x,y)&=\frac{1}{(1-x)(1-y)}-\underset{\geq}{\mathrm{\Omega}}\ \frac{1}{\underline{(1-\frac{x}{\lambda_2})}(1-y\lambda_2^L)}-\frac{y}{(1-xy^K)(1-y)}
\\&=\frac{1}{(1-x)(1-y)}-\frac{x}{(1-x)(1-yx^L)}-\frac{y}{(1-xy^K)(1-y)}
\\&=\frac{1+xy\frac{(1-x^L)(1-y^K)}{(1-x)(1-y)}-xy^K-x^Ly}{(1-xy^K)(1-x^Ly)}.
\end{align*}
This completes the proof.
\end{proof}

Our method remains applicable to several other results, including \cite[Theorems 3.1, 4.1, and 5.1]{Andrews2}, \cite[Lemma 3.1]{Andrews13}, and \cite[Page 139 (No. 386--394)]{MacMahon19}, among others. For brevity, we do not enumerate them all here.

\subsection{A Hard Problem}

We begin by recalling the definition of the \emph{$q$-shifted factorial}. For a complex number $a$ and a positive integer $n$, it is defined as
\[
(a;q)_0 = 1, \quad (a;q)_n = (1 - a)(1 - aq) \cdots (1 - aq^{n-1}).
\]
In \cite{Li-Liu-Xin}, the computation of the following generating function is a difficult problem and has not been solved:
\[
F(x,y,q) = \sum_{\substack{i \geq b_0 \geq b_1 \geq \cdots \geq b_k \geq 0; \\ b_0 \geq j \geq c_1 \geq c_2 \geq \cdots \geq c_r \geq 0}} x^i y^j q^{b_0 + b_1 + \cdots + b_k + c_1 + \cdots + c_r}.
\]
In this subsection, we employ MacMahon's partition analysis to derive an explicit formula for $F(x,y,q)$.

\begin{thm}
With the notation as above, we have
\begin{align*}
F(x,y,q)=\frac{1}{(y;q)_{r+1}(x;q)_{k+2}}+\sum_{i=0}^{r} \frac{(-1)^{i+1}yq^{i(i+3)/2}}{(1-x)(1-yq^i)(q;q)_i(q;q)_{r-i}(xyq^{i+1};q)_{k+1}}.
\end{align*}
\end{thm}
\begin{proof}
To establish this result, we calculate the following generating function:
$$G(x,y,t,z,q)=\sum_{i\geq b_0\geq b_1\geq \cdots\geq b_k\geq 0; \atop b_0\geq j\geq c_1\geq c_2\geq \cdots\geq c_r\geq 0}x^i y^j t^{b_0} z^{b_1} q^{b_2+\cdots+b_k+c_1+\cdots +c_r}.$$
Note that $F(x,y,q) = G(x,y,q,q,q)$.
Let
\begin{align*}
G_1(x,y,t,q)=\sum_{i\geq b_0\geq j\geq c_1\geq c_2\geq \cdots\geq c_r\geq 0}x^i y^j t^{b_0} q^{c_1+\cdots +c_r}
\ \ \ \ \text{and}\ \ \ \
G_2(z,q)=\sum_{b_1\geq \cdots\geq b_k\geq 0}z^{b_1} q^{b_2+\cdots+b_k}.
\end{align*}
It is straightforward to verify that
\begin{align*}
G_1(x,y,t,q)&=\frac{1}{(1-x)(1-xt)(1-xyt)(1-xytq)\cdots (1-xytq^r)},
\\ G_2(z,q)&=\frac{1}{(1-z)(1-zq)\cdots(1-zq^{k-1})}.
\end{align*}
By Definition \ref{dfn-natural} of the operator $\mathrm{\Omega}_\geq$, the generating function $G(x,y,t,z,q)$ can be expressed as
$$G(x,y,t,z,q)=\underset{\geq}{\mathrm{\Omega}}\ G_1(x,y,t\lambda,q)\cdot G_2\left(\frac{z}{\lambda},q\right).$$
Applying the operator step by step, we eliminate the variable $\lambda$ and obtain:
\begin{small}
\begin{align*}
&G(x,y,t,z,q)
\\=&\underset{\geq}{\mathrm{\Omega}}\ \frac{1}{(1-x)\underline{(1-xt\lambda)(1-xyt\lambda)(1-xytq\lambda)\cdots (1-xytq^r\lambda)}(1-\frac{z}{\lambda})(1-\frac{zq}{\lambda})\cdots(1-\frac{zq^{k-1}}{\lambda})}
\\=&\frac{1}{(1-x)(1-xt)(1-y)(1-yq)\cdots (1-yq^r)(1-xzt)(1-xztq)\cdots (1-xztq^{k-1})}
\\ \ &+\frac{1}{(1-x)(1-\frac{1}{y})(1-xyt)(1-q)\cdots (1-q^r)(1-xyzt)(1-xyztq)\cdots (1-xyztq^{k-1})}
\\ \ &+\frac{1}{(1-x)(1-\frac{1}{yq})(1-\frac{1}{q})(1-xytq)(1-q)\cdots (1-q^{r-1})(1-xyztq)(1-xyztq^2)\cdots (1-xyztq^{k})}
\\ \ &+\frac{1}{(1-x)(1-\frac{1}{yq^2})(1-\frac{1}{q^2})(1-\frac{1}{q})(1-xytq^2)(1-q)\cdots (1-q^{r-2})(1-xyztq^2)\cdots (1-xyztq^{k+1})}
\\ \ &+\cdots
\\ \ &+\frac{1}{(1-x)(1-\frac{1}{yq^r})(1-\frac{1}{q^r})(1-\frac{1}{q^{r-1}})\cdots (1-\frac{1}{q})(1-xyztq^r)(1-xyztq^{r+1})\cdots (1-xyztq^{k+r-1})}
\\=& \frac{1}{(1-x)(1-xt)(y;q)_{r+1}(xzt;q)_k}
+\sum_{i=0}^r\frac{(-1)^{i+1}yq^{\frac{i(i+3)}{2}}}{(1-x)(1-yq^i)(1-xytq^{i})(q;q)_i(q;q)_{r-i}(xyztq^i;q)_k}.
\end{align*}
\end{small}
Substituting $t = q$, $z = q$, and recalling that $F(x,y,q) = G(x,y,q,q,q)$, we arrive at the desired expression. This completes the proof.
\end{proof}

\subsection{Two Generating Functions}

Let $\mathbb{H}$ be a ring, such as the ring $\mathbb{Z}$ of integers. Consider the bivariate generating function
\[
F(x,y)=\sum_{i\geq j\geq 0}a_{ij}x^iy^j, \quad a_{ij}\in \mathbb{H}.
\]
Define the generating functions $G_1(x,y,z)$ and $G_2(x,z)$ by
\begin{align*}
G_1(x,y,z):=\sum_{m\geq i\geq \ell\geq j\geq 0}a_{ij}x^m y^{\ell} z^{i+j},\ \ \ \ \ G_2(x,z):=\sum_{i\geq \ell\geq j\geq 0}a_{ij}x^{\ell} z^{i+j}.
\end{align*}
As highlighted in \cite{Li-Liu-Xin}, these functions are of particular interest. We now derive their expressions via MacMahon's partition analysis.

\begin{thm}
Following the above notation, we have
\begin{align*}
G_1(x,y,z)=\frac{F(xz,yz)-yF(xyz,z)}{(1-x)(1-y)},\quad \quad
G_2(x,z)=\frac{F(z,xz)-xF(xz,z)}{1-x}.
\end{align*}
\end{thm}
\begin{proof}
By the definition of $\mathrm{\Omega}_\geq$ and the geometric series summation, we obtain
\begin{align*}
G_1(x,y,z)&=\underset{\geq}{\mathrm{\Omega}}\ \sum_{m,i,\ell,j\geq 0}a_{ij}\lambda_1^{m-i}\lambda_2^{i-\ell}\lambda_3^{\ell-j}x^my^{\ell}z^{i+j}
\\&=\underset{\geq}{\mathrm{\Omega}}\ \sum_{i,j\geq 0}a_{ij}\left(\frac{z\lambda_2}{\lambda_1}\right)^i\left(\frac{z}{\lambda_3}\right)^j\frac{1}{(1-x\lambda_1)(1-\frac{y\lambda_3}{\lambda_2})}
\\&=\underset{\geq}{\mathrm{\Omega}}\ F\left(\frac{z\lambda_2}{\lambda_1},\frac{z}{\lambda_3}\right)\frac{1}{(1-x\lambda_1)(1-\frac{y\lambda_3}{\lambda_2})}.
\end{align*}
The operator $\mathrm{\Omega}_{\geq}$ eliminates the variables $\lambda_1$, $\lambda_2$, and $\lambda_3$. We first eliminate $\lambda_1$ and $\lambda_3$, then $\lambda_2$. Applying this procedure:
\begin{align*}
G_1(x,y,z)&=\underset{\geq}{\mathrm{\Omega}}\ F\left(\frac{z\lambda_2}{\lambda_1},\frac{z}{\lambda_3}\right)\frac{1}{\underline{(1-x\lambda_1)}(1-\frac{y\lambda_3}{\lambda_2})}
=\underset{\geq}{\mathrm{\Omega}}\ F\left(xz\lambda_2,\frac{z}{\lambda_3}\right)\frac{1}{(1-x)\underline{(1-\frac{y\lambda_3}{\lambda_2})}}
\\&=\underset{\geq}{\mathrm{\Omega}}\ F\left(xz\lambda_2,\frac{yz}{\lambda_2}\right)\frac{1}{(1-x)(1-\frac{y}{\lambda_2})}.
\end{align*}
Note that the power series expansion of $F(xz\lambda_2, \frac{yz}{\lambda_2})$ contains only non-negative powers of $\lambda_2$. Considering the dual contribution in $\lambda_2$, we derive
\begin{align*}
G_1(x,y,z)&=\underset{\geq}{\mathrm{\Omega}}\ F\left(xz\lambda_2,\frac{yz}{\lambda_2}\right)\frac{1}{(1-x)\underline{(1-\frac{y}{\lambda_2})}}
=\frac{F(xz,yz)}{(1-x)(1-y)}-\frac{\lambda_2\cdot F(xz\lambda_2,\frac{yz}{\lambda_2})}{(1-x)\underline{(\lambda_2-y)}}
\\&=\frac{F(xz,yz)-yF(xyz,z)}{(1-x)(1-y)}.
\end{align*}
Similarly, for $G_2(x,z)$, we obtain
\begin{align*}
G_2(x,z)&=\underset{\geq}{\mathrm{\Omega}}\ \sum_{i,\ell,j\geq 0}a_{ij}\lambda_1^{i-\ell}\lambda_2^{\ell-j}x^{\ell}z^{i+j}
=\underset{\geq}{\mathrm{\Omega}}\ \sum_{i,j\geq 0}a_{ij}(z\lambda_1)^i\left(\frac{z}{\lambda_2}\right)^j\frac{1}{(1-\frac{x\lambda_2}{\lambda_1})}
\\&=\underset{\geq}{\mathrm{\Omega}}\ F\left(z\lambda_1,\frac{z}{\lambda_2}\right)\frac{1}{(1-\frac{x\lambda_2}{\lambda_1})}
=\underset{\geq}{\mathrm{\Omega}}\ F\left(z\lambda_1,\frac{xz}{\lambda_1}\right)\frac{1}{(1-\frac{x}{\lambda_1})}
\\&=\frac{F(z,xz)-xF(xz,z)}{1-x}.
\end{align*}
This completes the proof.
\end{proof}




\noindent
{\small \textbf{Acknowledgments:}
This work is partially supported by the National Natural Science Foundation of China [12571355].

\end{document}